\documentclass{amsart}

\usepackage{graphicx,psfrag,color}
\usepackage{caption}
\usepackage{subcaption,overpic}

\newtheorem{theorem}{Theorem}[section]

\newtheorem{lemma}[theorem]{Lemma}
\newtheorem{claim}{Claim}
\theoremstyle{definition}
\newtheorem{definition}[theorem]{Definition}
\newtheorem{question}[theorem]{Question}
\newcommand{\Down}{\Delta}
\newcommand{\Up}{\Upsilon}

\theoremstyle{remark}
\newtheorem*{claim*}{Claim}
\DeclareMathOperator{\pw}{pw}
\DeclareMathOperator{\tw}{tw}
\newcommand{\set}[1]{\ensuremath{\left\{ #1 \right\}}}

\title{Posets with cover graph of pathwidth two have bounded dimension}

\author[Bir\'o]{Csaba Bir\'o}
\address{Department of Mathematics, University of Louisville,
  Louisville, KY 40292}
\email{csaba.biro@louisville.edu}
\author[Keller]{Mitchel T.~Keller}
\address{Department of Mathematics, Washington and Lee University,
  Lexington, VA 24450}
\email{kellermt@wlu.edu}
\author[Young]{Stephen J.~Young}
\address{Department of Mathematics, University of Louisville,
  Louisville, KY 40292}
\email{stephen.young@louisville.edu}
\date{28 May 2015}
\begin{document}

\begin{abstract}
  Joret, Micek, Milans, Trotter, Walczak, and Wang recently asked if
  there exists a constant $d$ such that if $P$ is a poset with cover
  graph of $P$ of pathwidth at most $2$, then $\dim(P)\leq d$. We
  answer this question in the affirmative by showing that $d=17$ is
  sufficient. We also show that if $P$ is a poset containing the
  standard example $S_5$ as a subposet, then the cover graph of $P$
  has treewidth at least $3$.
\end{abstract}

\maketitle

\section{Introduction}

Although the dimension of a poset and the treewidth of a graph have
been prominent subjects of mathematical study for many years, it is
only recently that the impact of the treewidth of graphs on poset
dimension has received any real attention. This new interest in
connections between these topics has led to recasting an old result in
terms of treewidth. It is natural to phrase the following result from
1977 in terms of treewidth, which had been defined (using a different
name) by Halin in \cite{Hal-76} a year earlier. However, the
importance of treewidth (and the use of that name) only became widely
known through the work of Robertson and Seymour \cite{Rob-Sey-84}
nearly a decade
later. % the following 36-year-old result, which came only a year after
% Halin \cite{Hal-76} defined what would eventually become widely-known
% as treewidth through the work of Robertson and Seymour
% \cite{Rob-Sey-84} nearly a decade later, being recast in terms of
% treewidth.

\begin{theorem}[Trotter and Moore \cite{Tro-Moo-77}]\label{T:treedim}
  If\/ $P$ is a poset such that the cover graph of\/ $P$ is a tree,
  then\/ $\dim(P) \leq 3$. Equivalently, if\/ $P$ is a poset such that
  the cover graph of\/ $P$ is connected and has treewidth at most\/
  $1$, then\/ $\dim(P)\leq 3$.
\end{theorem}

Recently there have been a number of papers on the dimension of planar
posets \cite{Fel-Li-Tro-10,Fel-Tro-Wie-13-u,Str-Tro-14}. This work
naturally led to the question of bounding a poset's dimension in terms
of the treewidth of its cover graph. Over $30$ years ago, Kelly showed
in \cite{Kel-81} that there are planar posets having arbitrarily large
dimension by constructing a planar poset containing $S_d$, the
standard example of dimension $d$, as a subposet. These examples use
large height to stretch out $S_d$ to allow a planar embedding. Joret
et al.\ \cite{Jor-Mic-Mil-Tro-Wal-Wan-13-u} point out that the
pathwidth of Kelly's examples is $3$ for $d\geq 5$. Thus, any bound on
dimension solely in terms of pathwidth or treewidth is
impossible. However, they were able to show that it suffices to add a
bound on the height in order to bound the dimension. In particular,
they proved the following:

\begin{theorem}[Joret et al.\ \cite{Jor-Mic-Mil-Tro-Wal-Wan-13-u}]
  For every pair of positive integers\/ $(t,h)$, there exists a least
  positive integer\/ $d=d(t,h)$ so that if\/ $P$ is a poset of height at
  most\/ $h$ and the treewidth of the cover graph of\/ $P$ is at most\/ $t$,
  then\/ $\dim(P)\leq d$.
\end{theorem}

Motivated by the observation about the pathwidth of Kelly's examples,
Joret et al.\ concluded their paper by asking if there is a constant
$d$ such that if $P$ is a poset whose cover graph has pathwidth at
most $2$, then $\dim(P)\leq d$. They also asked this question with
treewidth replacing pathwidth. (An affirmative answer to the latter
question would imply an affirmative answer to the former.)  In this
paper, we show that the answer for pathwidth $2$ is in fact ``yes''
with the following result:

\begin{theorem}\label{T:weak-main}
  Let\/ $P$ be a poset. If the cover graph of\/ $P$ has pathwidth at most\/
  $2$, then \/$\dim(P)\leq 17$.
\end{theorem}

In fact, the precise version of this result (Theorem~\ref{T:full-main}) is
intermediate between answering the pathwidth question and
answering the treewidth question, as we only need to exclude six of
the $110$ forbidden minors that characterize the graphs of pathwidth
at most $2$. (Treewidth at most $2$ is characterized simply by
forbidding $K_4$ as a minor.)

We show in Theorem~\ref{T:S5} that any poset containing the standard
example $S_5$ has treewidth at least $3$. This provides a small piece
of evidence in favor of the idea that if the treewidth of a poset is
at most $2$, then the poset's dimension is bounded.

Before proceeding to our proofs, we provide some definitions for
completeness. We then establish some essential properties of the
$2$-connected blocks of a graph of pathwidth at most $2$. We then
prove the more general version of Theorem~\ref{T:weak-main} and
conclude with the rather technical proof that posets containing $S_5$
have cover graphs of treewidth at least $3$.

\section{Definitions and Pathwidth $2$ Obstructions}

Let $P$ be a poset. If $x<y$ in $P$ and there is no $z\in P$ such that
$x < z < y$ in $P$, we say that $x$ is covered by $y$ (or $y$ covers
$x$) and write $x<:y$. For $x\in P$, the closed down set of $x$ is
$D[x] = \set{y\in P\colon y\leq x}$ and the closed up set of $x$ is $U[x] =
\set{y\in P\colon y\geq x}$. The cover graph of $P$ is the graph $G$ with
the elements of $P$ as its vertices in which $x$ is adjacent to $y$ in
$G$ if and only if $x<:y$ or $y<:x$. (If we view the order diagram of
$P$ as a graph, that graph is $P$'s cover graph.) The dimension of $P$
is the least $t$ such that there exist $t$ linear
extensions---collectively known as a \emph{realizer}---$L_1,\dots,L_t$
of $P$ with the property that $x<_Py$ if and only if $x<_{L_i}y$ for
$i=1,\dots, t$. An incomparable pair $(x,y)$ of $P$ is said to be
\emph{reversed} by a linear extension $L$ if $y<_L x$. To show that a
set $\mathcal{R}$ of linear extensions of a poset $P$ is a realizer,
it suffices to show that each incomparable pair is reversed by some
linear extension in $\mathcal{R}$.
% We say that a set $S$ of ordered incomparable pairs of $P$ is
% \emph{reversible} if there exists a linear extension $L$ of $P$ such
% that $y<_L x$ for every $(x,y)\in S$. If $P$ is not a
% chain, the definition of dimension can be recast as being the least
% $t$ for which the set of incomparable pairs of $P$ can be partitioned
% into $t$ reversible sets. An \emph{alternating cycle} is a set
% $\set{(a_i,b_i)}_{i=1}^k$ of ordered incomparable pairs if
% $k\geq 2$ and $a_i\leq_P b_{i+1}$ for $i=1,2,\dots,k$, with subscripts
% interpreted cyclically. We say an alternating cycle is \emph{strict}
% provided that for each $i=1,2,\dots,k$, $a_i\leq_P b_j$ if and only if
% $j=i+1$. These notions are useful here since Trotter and Moore proved
% in \cite{Tro-Moo-77} that a set $S$ of ordered incomparable pairs is
% reversible if and only if $S$ contains no alternating cycles if and
% only if $S$ contains no strict alternating cycles.
% We say that $(x,y)$ is a critical pair of $P$ provided
% that $x$ and $y$ are incomparable in $P$, every $z\leq x$ in $P$ is
% also less than $y$ in $P$, and every $w\geq y$ in $P$ is also greater
% than $x$ in $P$. 
By the \emph{standard example} $S_n$, we mean the subposet of the
lattice of subsets of $\set{1,2,\dots,n}$ induced by the singletons
and the $(n-1)$ sets. For further background on the combinatorics of
partially ordered sets, refer to Trotter's monograph \cite{trotter:dimbook}.

% Felsner, Trotter, and Wiechert \cite{Fel-Tro-Wie-13-u} showed the
% following result on the dimension of posets with outerplanar cover
% graphs, which will prove essential to our proof.

% \begin{theorem}[Felsner, Trotter, Wiechert \cite{Fel-Tro-Wie-13-u}]
% \label{thm:outerplanar}
% If a poset
% $P$ has outerplanar cover graph, then $\dim(P)\leq 4$.
% \end{theorem}

Let $G=(V,E)$ be a graph. A pair $(T,\mathcal{V})$, where $T$ is a tree and
$\mathcal{V}=(V_t)_{t\in T}$ with $V_t\subseteq V$ for all $t\in T$, is a
\emph{tree-decomposition} of $G$ if
\begin{enumerate}
\item $V(G)$ is the union of all the $V_t$;
\item for every $e\in E$, there exists a vertex $t$ of $T$ such that
  $e\subseteq V_t$; and
\item if $t_1,t_2,t_3$ are vertices of $T$ and $t_2$ lies on the
  unique path from $t_1$ to $t_3$ in $T$, then $V_{t_1}\cap
  V_{t_3}\subseteq V_{t_2}$.
\end{enumerate}
The sets $V_t$ are often referred to as the \emph{bags} of the
tree-decomposition. The \emph{width} of $(T,\mathcal{V})$ is $\max_t
|V_t|-1$. The \emph{treewidth} of $G$, which we denote by $\tw(G)$, is the
minimum width of a tree-decomposition of $G$. A
\emph{path-decomposition} of a graph is a tree-decomposition in which
the tree $T$ is a path. The \emph{pathwidth} of $G$, denoted by $\pw(G)$, is
the minimum width of a path-decomposition of $G$.

Following Diestel \cite{diestel}, we make the following definition of
a special type of path to improve the readability of parts of our
argument. If $G$ is a graph and $H$ is a subgraph of $G$, we say that
a path $P$ is an \emph{$H$-path} if $P$ is nontrivial and intersects
$H$ precisely at its two end vertices. The length of a path is the
number of edges it contains. We will also freely use terminology
regarding the block structure of graphs. Readers unfamiliar with this
terminology should consult Diestel's text \cite{diestel}, in
particular Chapter 3.

By a \emph{subdivision} of a graph $G$ we mean a graph $G'$ in which
some edges of $G$ are replaced by paths that are internally disjoint
from each other and the vertices of $G$. The original vertices of $G$
are called the \emph{branch vertices} of $G'$. If a graph $H$ contains a
subdivision of $G$ as a subgraph, then we say that $G$ is a
\emph{topological minor} of $H$. An \emph{inflation} of a graph $G$ is
a graph $G'$ formed by replacing the vertices $x$ of $G$ by disjoint
connected graphs $G_x$ and the edges $xy$ of $G$ by nonempty sets of
edges from $G_x$ to $G_y$. The vertex sets $V(G_x)$ are called the
\emph{branch sets} of $G'$. If a graph $H$ contains an inflation of $G$
as a subgraph, we say that $G$ is a \emph{minor} of $H$. Equivalently,
$G$ is a minor of $H$ if $G$ can be obtained from $H$ by a sequence of
vertex deletions, edge deletions, and edge contractions. Note that if
the maximum degree of $G$ is at most $3$, the notions of minor and
topological minor are equivalent. For further information on minors
and topological minors, see Diestel's text \cite{diestel}.

The set of graphs of pathwidth at most $k$ is a minor closed
family. Therefore, by the Graph Minor Theorem \cite{Rob-Sey-04}, this
set of graphs can be characterized by forbidding a finite set of
graphs as minors. For $k=2$, Kinnersley and Langston found the entire
set of $110$ obstructions in \cite{Kin-Lan-94}. The proof of this
paper's main result relies on only six graphs from their list, but
having the whole list at hand was critical to the development of our
proof. Besides the obvious obstruction $K_4$, the other five we must
exclude are depicted in Figure~\ref{F:pw2-obstruct}. It is elementary
to verify that these graphs have pathwidth $3$. We will refer to these
graphs in the proof by the names shown and use $\mathcal{F}$ to denote
$\set{K_4,T_1,\ldots,T_5}$. If a graph $G$ does not contain an element
of $\mathcal{F}$ as a minor, we will say that $G$ is
\emph{$\mathcal{F}$-minor free}.

\begin{figure}
  \centering
  \begin{subfigure}[b]{0.33\textwidth}
    \centering
    \includegraphics[scale=0.35]{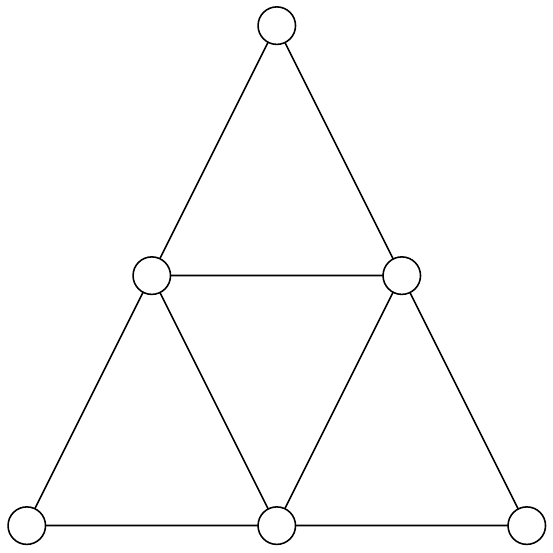}
    \caption{$T_1$}\label{6.4.1}
  \end{subfigure}%
  \begin{subfigure}[b]{0.33\textwidth}
    \centering
    \includegraphics[scale=0.35]{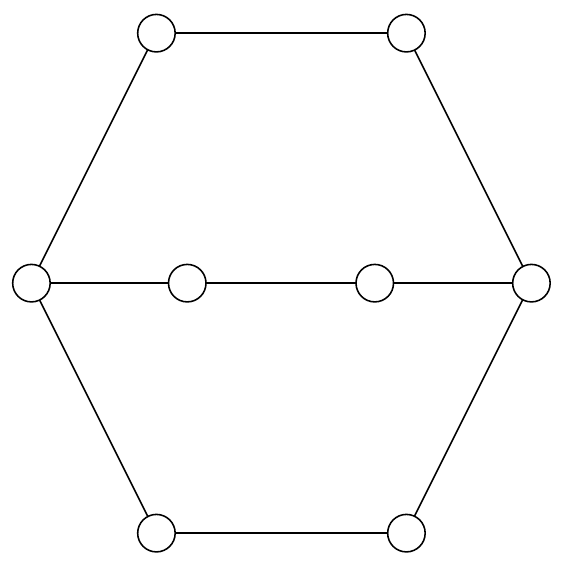}
    \caption{$T_2$}\label{8.2.1}
  \end{subfigure}%
  \begin{subfigure}[b]{0.33\textwidth}
    \centering
    \includegraphics[scale=0.35]{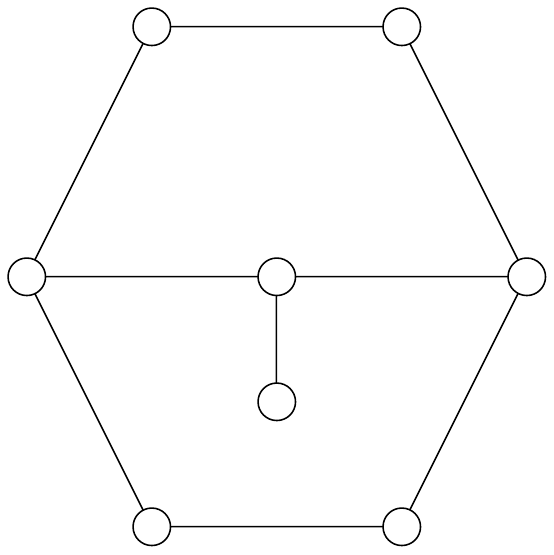}
    \caption{$T_3$}\label{8.2.2}
  \end{subfigure}

  \begin{subfigure}[b]{0.33\textwidth}
    \centering
    \includegraphics[scale=0.35]{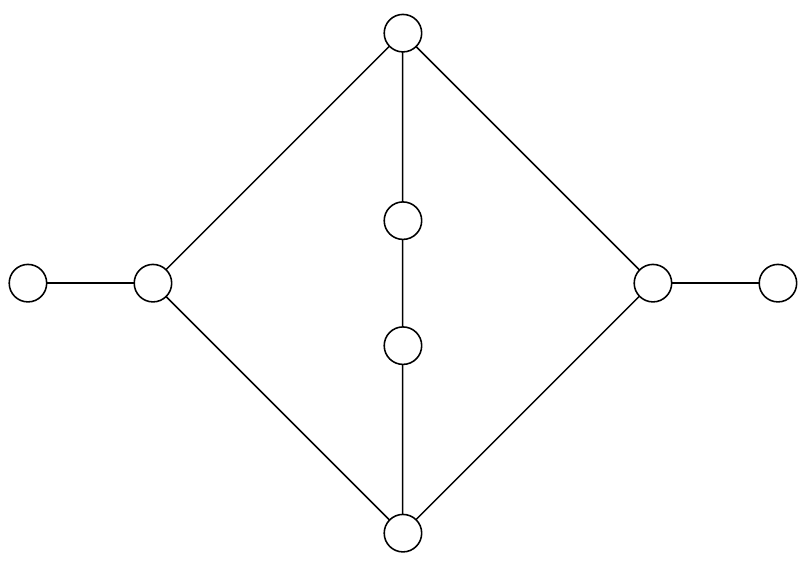}
    \caption{$T_4$}\label{8.2.3}
  \end{subfigure}\hspace{0.075\textwidth}
  \begin{subfigure}[b]{0.33\textwidth}
    \centering
    \includegraphics[scale=0.35]{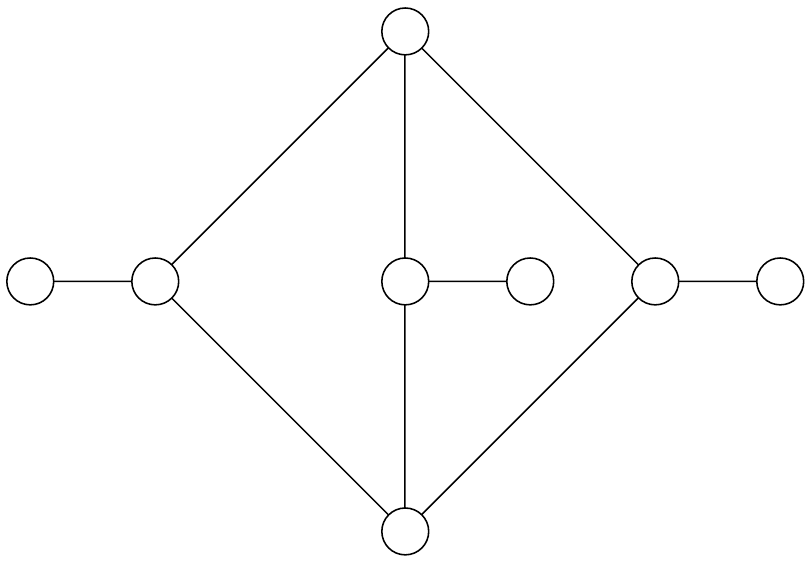}
    \caption{$T_5$}\label{8.2.4}
  \end{subfigure}

  \caption{Five key obstructions for pathwidth $2$.}\label{F:pw2-obstruct}
\end{figure}

\section{Properties of the $2$-connected blocks}

We begin without restricting our attention to only cover graphs. In
this section, we consider a graph $G$ such that $\pw(G)\leq 2$ and
prove strong properties of the block structure. This structure is
essential in the proof of our main theorem. To establish this
structural result, we first make the following definition.

\begin{definition}
  A \emph{parallel nearly outerplanar graph} is a graph that consists
  of a longest cycle $C$ with vertices labelled (in order) as
  $x_1,x_2,\ldots,x_k,y_l,y_{l-1},\ldots,y_1$ along with some chords
  and chords subdivided exactly once. The chords and subdivided chords
  have attachment points $x_{i_1},y_{j_1},\dots,x_{i_m},y_{j_m}$
  such that $i_1\leq\cdots\leq i_m$ and $j_1\leq\cdots\leq j_m$.
\end{definition}

An example of a parallel nearly outerplanar graph is shown in
Figure~\ref{F:codex}. We think of the vertices along the bottom of the
cycle as being the $x_i$ and those along the top as being the
$y_j$. Vertices to the left of the leftmost chord and to the right of
the rightmost chord could be either $x_i$'s or $y_j$'s.

\begin{figure}
\includegraphics[scale=0.5]{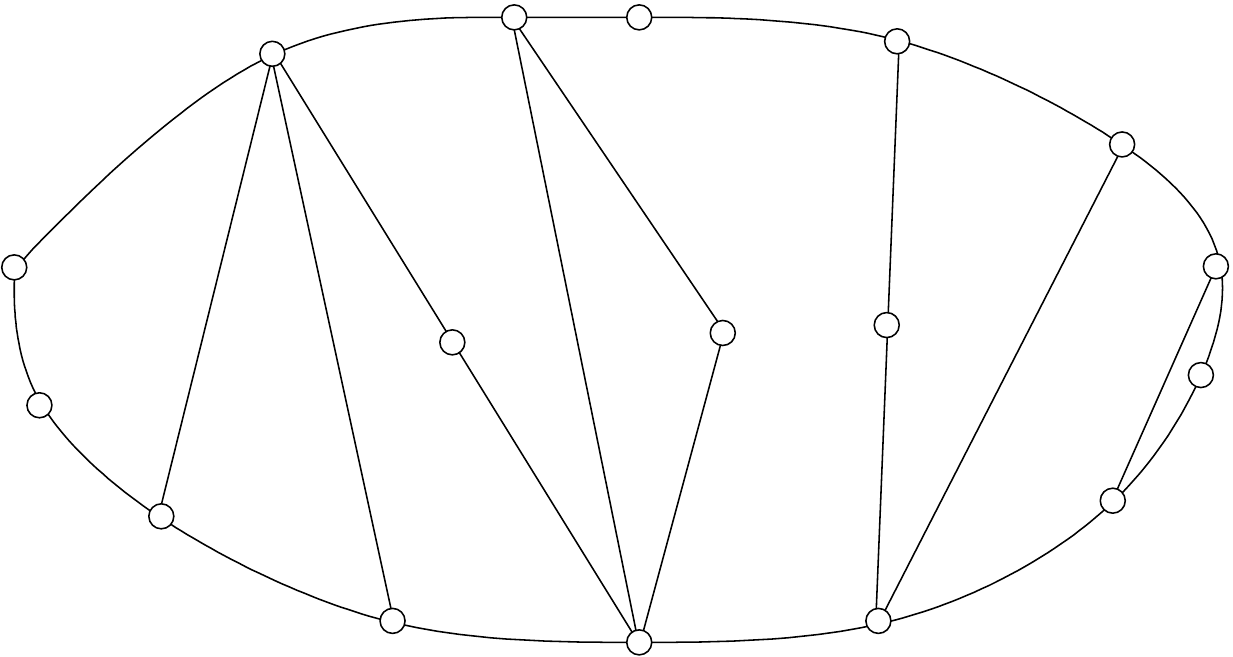}
\caption{A parallel nearly outerplanar graph.}\label{F:codex}
\end{figure}

\begin{lemma}
  A graph\/ $G$ is a parallel nearly outerplanar graph if and only if\/
   $G$ is\/ $2$-connected and\/ $\pw(G)\leq 2$.\label{lem:codex}
\end{lemma}

\begin{proof}
  It is easy to see that every parallel nearly outerplanar graph is
  $2$-connected and has pathwidth at most $2$. A path-decomposition of
  width $2$ can be obtained by starting with the bag containing $x_1$
  and $y_1$ and proceeding through the $x_i$ and $y_j$ by increasing
  subscript. After all edges incident with $x_i$ have had their other
  attachment point included in a bag with $x_i$, the bag
  $\set{x_i,x_{i+1},y_j}$, where $y_j$ is the ``current'' vertex from
  the other side of the cycle, covers the edge $x_ix_{i+1}$. We can
  then remove $x_i$ from the bag and continue. A symmetric process is
  used to move from $y_j$ to $y_{j+1}$ after covering all edges
  incident with $y_j$. The internal vertex of a subdivided chord
  appears in a bag with precisely its two attachment points.

  For the converse, let $C$ be a longest cycle in $G$. A $C$-path will
  be called an \emph{ear}. We first note that $C$ cannot have
  crossing ears. More precisely, if $P$ and $Q$ are ears, $V(C)\cap
  V(P)=\{p_1,p_2\}$, and $V(C)\cap V(Q)=\{q_1,q_2\}$, then the order
  of these intersection vertices on $C$ must be
  $p_i,p_{3-i},q_j,q_{3-j}$ for some $i,j\in\set{1,2}$. If this were not the
  case, then $G$ would have a $K_4$-minor, forcing $\pw(G)\geq 3$.

  Next we show that no ear may have more than one internal
  vertex. Indeed, if $P$ is an ear with at least two internal vertices
  and $V(C)\cap P=\{v_1,v_2\}$, then both paths between $v_1$ and
  $v_2$ on $C$ must contain at least two internal vertices, for
  otherwise $C$ is not the longest cycle. If this occurs, then $G$ has
  a $T_2$-minor.

  We now show that the internal vertex, if one exists, of any ear is
  of degree $2$. Let $v$ be the internal vertex of the ear $xvy$, and
  suppose that the degree of $v$ is at least $3$. Let
  $H$ be the subgraph induced by the vertices of $C$ and the vertex
  $v$. If $v$ has degree at least $3$ in $H$, then $H$ contains a
  $K_4$-minor. Otherwise, there is a $v'\in V(G)$ such that $v' v\in
  E(G)$, but $v'\not\in V(H)$. Let $H'$ be the subgraph of $G$ formed
  from $H$ by adding the vertex $v'$ and edge $vv'$. Since $G$ is
  $2$-connected and $H'$ is not, there is an $H'$-path $P$ (possibly
  just a single edge) with one endpoint being $v'$. The other endpoint
  may only be $x$ or $y$, since otherwise we have a $K_4$-minor. Without
  loss of generality, the other endpoint is $x$, which implies that
  $xPv'vy$ is an ear with at least two internal vertices, a
  contradiction.

  We have now shown that $G$ contains a (longest) cycle and some
  non-crossing ears with at most one inner vertex which must have
  degree two. The only thing that remains to be shown is that the
  vertices of the cycle may be labeled as in the definition,
  effectively placing an ordered structure on the ears. If this were
  not true, there would be three ears with attachment points
  $a_1,b_1$, $a_2,b_2$, and $a_3,b_3$ that appear around the longest
  cycle of $G$ ordered as $a_1,b_1,a_2,b_2,a_3,b_3$ around $C$, with
  the possibility that $b_i=a_{i+1}$ for any $i$ (cyclically). In this
  case, $G$ contains the forbidden minor $T_1$, which gives our
  final contradiction.
\end{proof}

We observe that our proof of the ``if'' direction of
Lemma~\ref{lem:codex} only requires that $G$ is $2$-connected and not
contain $K_4$, $T_1$, or $T_2$ as a minor. Furthermore, the cycle
bounding the infinite face may be chosen to be \emph{any} longest
cycle of the graph, a fact which we will use in the proof of Lemma~\ref{lem:long-cycle}.

We note that after proving Lemma~\ref{lem:codex}, we discovered that
Bar\'at et al.\ \cite{Bar-Haj-Lin-Yan-12} had previously proved this
fact while working to simplify the characterization of graphs of
pathwidth $2$. They used the name \emph{track} for what we call a
parallel nearly outerplanar graph. We use the latter name because it
is more evocative of the aspects of the structure that are important
in our proof and include the proof of Lemma~\ref{lem:codex} for
completeness.

By Lemma~\ref{lem:codex}, each $2$-connected block of a graph of
pathwidth two is a parallel nearly outerplanar graph. Our next lemma
establishes that the vertices where these blocks join together lie on
the parallel nearly outerplanar graphs' longest cycles.

\begin{lemma}\label{lem:long-cycle}
  Let\/ $G$ be a connected graph that does not contain an element of\/ 
  $\mathcal{F}$ as a minor. Let\/ $B$ be a\/ 
  $2$-connected block of\/ $G$. There exists a longest cycle\/ $C$ of\/ $B$
  such that if there is a vertex $v$ of $B$ adjacent
  to a vertex\/ $v'$ not in\/ $B$, then\/ $v\in V(C)$.
\end{lemma}

\begin{proof}
  Let $C$ be a longest cycle of $B$ that minimizes the number of
  internal vertices of ears adjacent to vertices outside $B$. Let $v$
  be an internal vertex of an ear $xvy$, i.e., $v\not\in V(C)$,  and suppose $v$ is adjacent a
  vertex $v'$ not in $B$. 
% Suppose
%   that $v\not\in V(C)$. This means that $v$ must be an internal vertex
%   of an ear $xvy$.
  Deletion of $x$ and $y$ from the cycle $C$ leaves two paths, which
  we will call $C_1$ and $C_2$. If both $C_1$ and $C_2$ contain at
  least two vertices, then $G$ has a $T_3$-minor, since we have
  assumed that $v$ has a neighbor $v'$ not in $B$. Thus, suppose $C_1$
  contains a single vertex $u$. If the degree of $u$ in $G$ is two,
  then the cycle formed from $C$ by replacing $u$ by $v$ is also a
  longest cycle of $B$, and has fewer internal vertices of ears
  adjacent to vertices outside $B$. If the degree of $u$ in $B$ were
  $3$, then there would be an ear $uzx$ or $uzy$. In either case, $C$
  would not be a longest cycle, as the edge $ux$ (or $uy$) could be
  replaced by the path $uzx$ (or $uzy$). Therefore, we may assume that
  $u$ is adjacent to a vertex $u'$ not in $B$. Furthermore, $u'\neq
  v'$, and there is no path from $u'$ to $v'$ in $G$ that does not go
  through $B$, as otherwise $G$ would contain a $K_4$-minor. If $C_2$
  contains at least two vertices, then $G$ contains a $T_4$-minor. If
  $C_2$ is a single vertex $w$, then it must have degree $2$ in $G$ to
  avoid having $T_5$ as a minor. But then $\{x,v,y,u\}$ is a longest
  cycle of $B$ with fewer internal vertices of ears adjacent to
  vertices outside $B$ than $C$.
%It must contain a vertex to make $C$ a longest cycle, but
%it's not important for the proof.
% Call the other arc $C_2$. The vertex $v$ has a neighbor in $V(B')$; call it
% $v'$. $G$ has a path from $v'$ to $C$, but it can only reach $C$ at $x$ or $y$,
% otherwise $G$ has a $K_4$ minor. Without loss of generality assume that the
% path reaches $C$ at $x$. Then the cycle $C_2 xv'vy$ is longer than $C$, a
% contradiction.
\end{proof}

In light of Lemma~\ref{lem:long-cycle}, we see that every
$\mathcal{F}$-minor free graph $G$ has a planar embedding in which each
$2$-connected block $B$ is embedded such that the vertices of $B$
lying on the unbounded face form a longest cycle of
$B$. We call such an embedding a \emph{canonical embedding} of $G$.

\section{Posets with Cover Graphs of Pathwidth $2$}
\begin{definition}
  Let $P$ be a poset. A \emph{subdivision} of the cover relation $x<:y$
  in $P$ is the addition of new points $z_1,z_2,\ldots,z_l$ such that
  $x<z_1<\dots<z_l<y$ and the new points $z_i$ are incomparable with
  all points of $P$ that are not greater than $y$ or less than $x$. We
  say that $Q$ is a \emph{subdivision} of $P$ if $Q$ can be
  constructed from $P$ by subdividing some of its cover relations.
\end{definition}

In light of what we know from the previous section about the structure
of graphs of pathwidth at most $2$, it is tempting to consider the
effect of subdivision on dimension. Since such an approach would allow
us to deal with some of the subdivided chords preventing the cover
graph from being outerplanar, we might be inclined to hope that if $Q$
is a subdivision of $P$, then $\dim(Q)\leq c\dim(P)$ for some absolute
constant $c$.  (Perhaps even $c=2$.) However, this is not the case. In
fact, Spinrad showed in \cite{spinrad:subdiv} that this construction
can increase dimension by an arbitrarily large factor. Fortunately, as
we show in Lemma~\ref{lem:subdivision}, there is a subdivision-like
operation on the graphs of relevance to our result that has a small
effect on the poset's dimension.

Our proof requires that we first introduce some additional
terminology. Let $G$ be a parallel nearly outerplanar graph that is
the cover graph of a poset $P$, and let $C$ be a longest cycle
provided by Lemma~\ref{lem:long-cycle}. An ear with no inner vertex is
simply called a \emph{chord}. We call an ear $xzy$ \emph{unidirected}
if $x<z<y$ or $y<z<x$ in $P$. Otherwise we call the ear a
\emph{beak}. An \emph{upbeak} is an ear with $x<z>y$ in $P$, and a
\emph{downbeak} is an ear with $x>z<y$ in $P$. (In either case,
$x\parallel y$.) We call the internal point of a beak a \emph{beak
  peak}. Our first step will be to address unidirected ears. We will
then turn our attention to the issue of beaks.

\begin{lemma}\label{lem:subdivision}
  Let\/ $P$ be a poset with cover graph\/ $G$. Suppose that\/ $G$ is\/
  $\mathcal{F}$-minor free and fix a canonical embedding of\/ $G$ in the
  plane. If\/ $Z$ is the collection of points that are not on the
  unbounded face of\/ $G$ and are neither minimal nor maximal in\/ $P$,
  then\/ $\dim(P) \leq 2 \dim(P-Z) + 1$.
\end{lemma}

\begin{proof}
  First notice that in a canonical embedding of $G$, our definition of $Z$ means that every element of
  $Z$ is the internal vertex of a unidirected ear $\ell<z<u$ in $P$. If
  the relation $\ell<u$ in $P-Z$ is a cover, then $z$ is a subdividing
  point of the cover relation $\ell<:u$ in $P$. Note, however, that $P$
  is not necessarily a subdivision of $P-Z$, as some of the
  unidirected ears may not correspond to cover relations in
  $P-Z$. Nevertheless, we will refer to $Z$ as the set of subdividing
  points of $P$ and an element of $Z$ will be called a subdividing
  point of $P$ even if the comparability involved is not a cover of
  $P$. When $\ell zu$ is a unidirected ear of $P$ with $\ell<z<u$ in $P$, we
  will refer to $\ell$ as the lower element of $z$. Similarly, $u$ will
  be called the upper element of $z$.

  Let $\{L_1,\dots,L_d\}$ be a realizer of $P-Z$ with
  $d=\dim(P-Z)$. For each $L_i$, we will construct two linear
  extensions $L_i'$ and $L_i''$ of $P$ by inserting the subdividing
  elements appropriately, and we will show that most incomparable
  pairs will be reversed in one of these linear extensions. We will
  create one extra linear extension to reverse the rest of the
  incomparable pairs.

  To construct $L_i'$, we place each subdividing point of $P$
  immediately above its lower element in $L_i$. We form $L_i''$ by
  placing each subdividing point immediately below its upper
  element in $L_i$. There may be some ambiguity in this definition if
  subdividing points share upper or lower elements. To deal with such
  situations, let $z_1,\ldots,z_k$ be subdividing points of $P$ that
  share the lower element $\ell$. For $j=1,\dots,k$, let the upper
  element of $z_j$ be $u_j$.  We may assume that these upper elements
  are distinct, since the removal of one point of a pair of points
  with duplicated holdings does not impact dimension (other than in
  the irrelevant case of a two-element antichain). Let $\sigma$ be a
  permutation of $\set{1,\dots,k}$ such that $u_{\sigma(1)}<\cdots<
  u_{\sigma(k)}$ in $L_i$. In $L_i'$ we insert the subdividing points
  so that $\ell< z_{\sigma(k)}<\cdots< z_{\sigma(1)}$. For $L_i''$, our
  concern is with subdividing points $z_1,\dots,z_k$ sharing the upper
  element $u$. Let $\ell_j$ be the lower element of $z_j$, and let
  $\sigma$ be a permutation of $\set{1,\dots,k}$ such that
  $\ell_{\sigma(1)}<\cdots< \ell_{\sigma(k)}$ in $L_i$. To form $L_i''$, we
  insert the subdividing elements so that $ z_{\sigma(k)}<\cdots<
  z_{\sigma(1)}<u$ in $L_i''$.

  Consider an incomparable pair $(a,b)$. If $a,b\in P-Z$, then obviously
  there is a linear extension $L_i'$ (and an $L_i''$) with $a>b$. Suppose
  $a\in P-Z$ and $b\in Z$ and let $\ell$ be the lower element of
  $b$. Then $a\not<\ell$ in $P-Z$ implies that there in an $L_i$ in which
  $a>\ell$, and hence $a>b$ in $L_i'$. Similarly, if $a\in Z$ and $b\in
  P-Z$, there exists an $L_i''$ with $a>b$.

  If $a,b\in Z$ have the same lower element, then their order in
  $L_i'$ will be opposite to their order in $L_i''$. Hence, one of
  $L_i'$ and $L_i''$ has $a>b$.  A similar argument works when $a$ and
  $b$ have the same upper element.

  Next we assume that $a,b\in Z$ have distinct upper and lower
  elements. Specifically, let $\ell_a$ and $u_a$ be the lower and
  upper elments of $a$ and let $\ell_b$ and $u_b$ be the
  lower and upper elements of $b$. If
  $\ell_a\not< \ell_b$, then $\ell_a>\ell_b$ in some $L_i$, and hence $a>b$ in
  $L_i'$. Similarly, if $u_a\not<u_b$, then $a>b$ in some $L_i''$.

  At this stage, we have shown that the incomparable pair $(a,b)$ will
  be reversed, unless all of the following conditions are satisfied:
  \begin{enumerate}
  \item $a,b\in Z$;
  \item $a$ and $b$ have distinct lower elements $\ell_a$ and $\ell_b$, respectively, and distinct upper elements $u_a$ and $u_b$, respectively; and
  \item $\ell_a<\ell_b$ and $u_a<u_b$.
  \end{enumerate}
  We say such a pair $(a,b)$ is \emph{in a bad diamond}. We will prove
  that there exists a single linear extension that reverses all such
  pairs.
  
  We do this by viewing the poset $P$ as an acyclic directed graph
  $D$, with directed edges corresponding to covers and pointing from
  smaller elements to larger elements.  For each incomparable pair
  $(a,b)$ in a bad diamond, we introduce a new directed edge $ba$. We
  call these \emph{new edges}, and the directed graph formed from $D$
  by adding these new edges is denoted by $D'$. Note that $a$ and $b$ must lie in the same $2$-connected block, so the new edge $ba$ will be added to within that block.

  The goal of the rest of the argument is to prove that $D'$ contains no
  oriented cycles.  Recall that we have fixed a canonical embedding of $D$ in the plane, which defines
  (up to duality) a natural linear order on the subdividing points. We fix one
  of these orders and use the terms ``left'' and ``right'' to refer to
  directions in this linear order. For upper and lower elements of the
  subdivided chords there is also a natural notion of two sides of the outer
  cycle defined by the embedding, depending on whether they are $x_i$'s or $y_j$'s. (This notion is well-defined, since we are concerned only with attachment points of subdivided chords.)
 
  \begin{claim}\label{claim:separation}
  Let $ba$ be a new edge. Then there is a directed path $P_{\ell}$ from $\ell_a$ to
  $\ell_b$, and a directed path $P_u$ from $u_a$ to $u_b$ in $D$, and for any such
  directed paths we have $P_{\ell}\cap P_u=\emptyset$, and in particular,
  $u_a,u_b\not\in P_{\ell}$ and $\ell_a,\ell_b\not\in P_u$.
  \end{claim}

  \begin{proof}
  The existence of the paths follows from condition (3) of the definition of
  bad diamonds.  If there exists $x \in P_{\ell} \cap P_u$, then we
  have that $a < u_a \leq x \leq \ell_b < b$, a contradiction.
  \end{proof}
  
  \begin{claim}\label{claim:sameside}
  Let $ba$ be a new edge. Then $\ell_a$ and $\ell_b$ are on the
  same side of the outer cycle, and $u_a$ and $u_b$ are also on the same side.
  Furthermore, $P_{\ell}$ and $P_u$ are on the outer cycle.
  \end{claim}

  \begin{proof}
  This is direct consequence of Claim~\ref{claim:separation}.
  If any part of the statement is not true, then $P_{\ell}$ topologically separates
  $u_a$ from $u_b$ or $P_u$ topologically separates $\ell_a$ from $\ell_b$.
  \end{proof}

  \begin{claim}\label{claim:orient}
  Let $cb$ and $ba$ be two new edges. Then they both go left, or both go right.
  \end{claim}

  \begin{proof}
  Without loss of generality assume for a contradiction that $ba$ goes left,
  and $cb$ goes right.  By Claim~\ref{claim:sameside}, all of $\ell_a$, $\ell_b$,
  $\ell_c$ are on the same side, and $u_a$, $u_b$, $u_c$ are on the same side.
  Furthermore, every directed path $P_{ab}$ from $\ell_a$ to $\ell_b$ goes on the
  outer cycle; a similar statement holds for paths $P_{bc}$ from $\ell_b$ to $\ell_c$.
  However, one of these is a subpath of the other, and they are directed
  contradictorily. 
  \end{proof}
  
  Now we are ready to show that $D'$ does not contain a directed cycle. Suppose
  for a contradiction that it does, and let $C$ be a directed cycle in $D'$
  that contains as few new edges as possible. Notice that $C$ must contain at
  least one new edge and at least one old edge by Claim~\ref{claim:orient}. Let $P_1$ be a maximal path in
  $C$ that consists entirely of new edges. Suppose that $P_1$'s initial point
  is $b$ and its terminal point is $a$. Notice that $C$ must lie entirely
  within a $2$-connected block of $D$, and this block is parallel nearly
  outerplanar. Also notice that $C$ must include the edges $au_a$ and $\ell_bb$,
  and a directed path $P_2$ from $u_a$ to $\ell_b$ that is disjoint
  from $P_1$.  For any $x, y \in P_2$ denote by $xP_2y$ the subpath of
  $P_2$ starting with $x$ and terminating with $y$.
  If $\ell_a\in P_2$, then the directed cycle $\ell_aa(u_aP_2\ell_a)$ contains fewer new
  edges than $C$; if $u_b\in P_2$, then $\ell_bb(u_bP_2\ell_b)$ is such a cycle.

  Therefore $P_2$ connects the unidirected ears $\ell_a a u_a$ and $\ell_b b u_b$. Hence
  $P_2$ must cross from the side of $u_a$ to the side of $\ell_b$.
  This must occur via a chord or a unidirected ear.  Let $u_0$ be the
  attachment point for the chord or unidirected ear on the same side as $u_a$ and let $\ell_0$ be the
  attachment point on the same side as $\ell_a$.   As all the
  new edges which form $P_1$ are all consistently oriented, this
  crossing
  occurs between some $a'$ and $b'$ which are consecutive vertices on
  $P_1$.  Since $b'a'$ is a new edge, we have that $(a',b')$ is in a bad
  diamond and in particular, $a'$ is incomparable to $b'$.  However,
  by Claim~\ref{claim:sameside} and the definition of a bad diamond, we
  have that $a' < u_{a'} \leq u_0 < \ell_0 \leq \ell_{b'} < b'$, a contradiction.

  Since $D'$ is acyclic, there is a total order $L_0$ on its vertices
  that respects the orientation of its edges. By construction, $L_0$
  is then a linear extension of $P$ that reverses all incomparable
  pairs that are in bad diamonds. Therefore, we can conclude that
  $\set{L_0,L_1',\dots,L_d',L_1'',\dots,L_d''}$ is a realizer of $P$
  and $\dim(P)\leq 2\dim(P-Z)+1$.
\end{proof}

To address the case of beaks in the cover graph, we will form two
extensions of the poset and show that their intersection is
$P-Z$. (Recall that $Z$ is the set of vertices that, in a canonical
embedding of $G$, are not on the unbounded face and are not beak
peaks.) We will then apply Lemma~\ref{lem:subdivision} to $P$ and use
what we know about the extensions of $P-Z$ to bound its
dimension. Note that in the remainder of this section, we often view
the poset as a directed graph and refer to a chain of covers as a
directed path.

\begin{lemma}\label{lem:up-down}
  Let\/ $P$ be a poset with cover graph\/ $G$. If\/ $G$ is\/
  $\mathcal{F}$-minor free, then\/ $P$ has extensions\/ $\Up$ and\/
  $\Down$ with\/ cover graphs\/ $G_\Up$ and\/
  $G_\Down$ that are outerplanar except for some chords replaced by
  directed paths of length $2$.
\end{lemma}

\begin{proof}
  Fix a canonical embedding of $G$. To construct $\Up$ and $\Down$, we
  consider the $2$-connected blocks of the cover graph of $P$ one at a
  time. In each block, we consider the beaks $xzy$ and introduce a
  comparability between $x$ and $y$. It is clear that if we are able to do
  this, beaks in $G$ will become edges in $G_\Up$ and $G_\Down$ and a
  pendant vertex (corresponding to the beak peak) will be added to one
  of the beak attachment points. Thus, the only obstruction to $G_\Up$
  and $G_\Down$ being outerplanar will come from unidirected ears,
  corresponding to replacing chords of an outerplanar graph by
  directed paths of length $2$.

  We introduce comparabilities between beak attachment points for all beaks
  in such a way that we maintain consistency of these new
  comparabilities. Since two blocks intersect in at most one point on
  their longest cycles, introducing a new comparability within one
  block cannot force two incomparable beak attachment points in
  another block to become comparable by transitivity.  Therefore, we
  may define the extensions on the blocks independently.

    Consider a 2-connected block $B$. Since $B$ is parallel nearly
    outerplanar, a fixed plane embedding provides (up to duality) a
    natural left-to-right ordering on its beaks as suggested in
    Figure~\ref{F:codex}. Fix one of these orders and number the $k$
    beaks of $B$ accordingly from $1$ to $k$. Denote the attachment
    points for beak $i$ by $x_i$ and $y_i$, with the $x_i$ all lying
    on the same side of the outer cycle of $B$ and the $y_i$ lying on
    the other.

    We now show that there exists an extension of the subposet induced
    by the vertices of $B$ in which $x_i<y_i$ for all $i=1,\ldots,k$.
    Let $D$ be the digraph defined by the subposet induced by the
    vertices of $B$.  More specifically, $V(D)=V(B)$ and there is a
    directed edge $uv$ in $D$ if and only if $(u,v)$ is a cover in
    $P$. To prove that  such an extension exists, it suffices to show that
    if we construct $D'$ by adding the directed edges $x_i y_i$ to
    $D$, then $D'$ contains no directed cycle. By a slight abuse of
    terminology, we will call the added directed edge $x_i y_i$ a
    beak.

    Suppose for a contradiction that $D'$ contains a directed cycle
    $C'$. Notice that $C'$ must contain at least one beak, because $D$
    is an acyclic graph. In fact, $C'$ has to contain at least two
    beaks, for if the only beak it contains were $x_i y_i$, then
    $y_i<x_i$ in $P$, which would contradict the fact that $x_iy_i$ is
    a beak. Therefore, $C'$ contains the beaks $x_i y_i$ and $x_j
    y_j$. As a consequence, $C'$ must contain a directed path between
    $y_i$ and $x_j$. This path forces $x_i$ and $y_j$ to belong to
    different (topological) regions, contradicting the existence of
    $C'$ as a directed cycle.  

    By a symmetric argument, there exists an extension of the subposet
    induced by the vertices of $B$ in which
    $y_i<x_i$ for all $i=1,\ldots,k$.

    Now we can define the extensions $\Up$ and $\Down$ of $P$. In a
    given embedding with left-right orientations of the 2-connected
    blocks, construct $\Up$ by adding, for each block, the relations
    $x_i<y_i$ for all $i$. Similarly, construct $\Down$ by adding the
    relations $y_i<x_i$ for all $i$ in each block.
\end{proof}

The final major step in our argument is to prove that
$P=\Up\cap\Down$, as then we may use realizers of $\Up$ and $\Down$ to
construct a realizer of $P$, thereby bounding the dimension.

\begin{lemma}\label{lem:p-is-up-down}
  Let\/ $P$ be a poset with\/ $\mathcal{F}$-minor-free cover
  graph. If\/ $\Up$ and\/ $\Down$ are extensions of\/ $P$ as defined
  in the proof of Lemma~\ref{lem:up-down}, then\/ $P = \Up\cap\Down$.
\end{lemma}

\begin{proof}
  It is sufficient to show that if $w\not<w'$ in $P$, then one of the
  extensions preserves this (non)relation.  We begin by considering
  the situation where $w$ and $w'$ are in the same $2$-connected block of the cover
  graph. We first address the case where $w$ and $w'$ are both on the
  outer cycle of a $2$-connected block and then reduce the remaining
  cases to this one. We conclude by addressing what happens when $w$ and
  $w'$ are in different blocks.
  
  \textbf{Case I} Suppose $w$ and $w'$ are both on the outer cycle $C$
  of a $2$-connected block $B$ and that $w<w'$ in both $\Up$ and
  $\Down$. There are directed paths (chains) from $w$ to $w'$ in both
  $\Up$ and $\Down$. We consider the shortest of these paths in the
  sense of containing the fewest beaks. Let $x_i y_i$ be the last beak
  on the path in $\Up$, and $y_j x_j$ be the last beak on the path in
  $\Down$. If $y_i = y_j$, then since $y_i < w'$ in $P$, there is a
  shorter path in $\Down$ that skips $y_jx_j$. Thus $y_i\neq y_j$. For
  a similar reason, $x_i\ne x_j$.

  Without loss of generality, assume that $i<j$. Suppose $w'$ is right of
  $y_j x_j$ (allowing $w'=x_j$) and consider a path in $P$ from
  $y_i$ to $w'$. By minimality, this path cannot pass through $y_j$,
  because then $y_j x_j$ could be skipped. Hence, the path separates
  $x_i$ from $y_j$. Notice that $w$ is not on the path from $y_i$ to
  $w'$, as this would imply $w<w'$ in $P$. Therefore, $w$ would have
  to be in both (topological) regions, which is a contradiction.  A
  similar contradiction can be derived if $w'$ is left of $x_i y_i$ or
  $w'=y_i$. In that case the path from $x_j$ to $w'$ in $P$ would
  separate $x_i$ from $y_j$.

  This leaves only the possibility that $w'$ is between the two
  beaks. If $w'$ is on the $x_i x_j$ arc of the outer cycle, then the
  path from $y_i$ to $ w'$ separates $x_i$ from $y_j$, and if $w'$ is
  on the path from $y_i$ to $ y_j$, then the $x_j w'$ path performs
  the separation. Therefore, we may conclude that $w\not<w'$ in $\Up$
  or $\Down$.

  \textbf{Case II} Still assuming $w$ and $w'$ are in the same
  $2$-connected block, we now suppose that exactly one of them is on
  the outer cycle $C$. Specifically, we will consider the case when
  $w$ is on $C$ and $w'$ is not, and the ear conatining $w'$ is right
  of $w$. This is just for convenience of discussion; the other three
  possibilities have identical proofs.

  Suppose there is a directed path from $w$ to $w'$ in both $\Up$ and
  $\Down$; consider one of these that goes through the minimum number
  of (newly-directed) beaks.  First note that $w'$ cannot be a peak of
  a downbeak, since that would make $w'$ minimal in $P$ and thus in
  $\Up$ and $\Down$. If $w'$ is a subdividing point of a unidirected ear,
  then let $u<w'$ be its attachment point. We have $w\not< u$ in $P$,
  so by Case I, we maintain this in one of $\Up$ or $\Down$. That
  extension preserves $w\not< w'$.

  The remaining possibility in this case is that $w'$ is the peak of
  an upbeak. By the minimality of the path from $w$ to $w'$, the path
  uses no beaks right of the beak containing $w'$. For the purpose of
  the argument, we may ignore all ears, chords, and points of $C$
  strictly right from the beak of $w'$. By so doing, $w'$ becomes a
  point on the outer cycle, and by Case I, one of $\Up$ or $\Down$
  will preserve $w\not< w'$.

  \textbf{Case III} To conclude the scenario where both $w$ and $w'$
  are in the same $2$-connected block, it remains only to address the
  case when neither of them is on $C$.  Without loss of generality
  assume that $w$ is left of $w'$. Considering a path from $w$ to $w'$
  in $\Upsilon$ or $\Delta$
  through the fewest number of beaks, we may assume that this path
  does not touch any part of the block left of $w$ and right of
  $w'$. (If either $w$ or $w'$ is part of a unidirected ear, using
  these portions would imply the existence of a directed cycle, and
  for beak peaks the path can be shortened by going via the other
  attachment point.) By ignoring the parts of the block left of $w$
  and right of $w'$, we place $w$ and $w'$ on an outer cycle, and thus
  Case I guarantees one of $\Up$ and $\Down$ preserves $w\not< w'$.

  \textbf{Case IV} It remains only to consider the case where no
  $2$-connected block contains both $w$ and
  $w'$.  If $w$ and $w'$ lie in
  different components of the cover graph, both $\Up$ and $\Down$
  preserve $w\not< w'$. Hence, we may assume there exists a path in
  the cover graph from $w$ to $w'$. (Since $w\not < w'$ in $P$, this
  path is \emph{not} a directed path.) Let the $2$-connected blocks
  containing an edge of the path be called $B_1,B_2,\ldots,B_l$. Note that we allow $l=0$ if the path does not pass
  through any 2-connected blocks, in which case $\Upsilon$ and
  $\Delta$ do not introduce comparabilities that could make $w$ and
  $w'$ comparable. Let $a_i$ and $b_i$ be the
  (uniquely-determined) entry and exit vertices of the path into and
  out of $B_i$; if $w\in B_1$, then let $a_1=w$, and if $w'\in B_l$,
  then let $b_l=w'$.

  If $a_i\le b_i$ in $P$ for all $i=1,2,\ldots,l$, then since the path
  from $w$ to $w'$ in the cover graph of $P$ is not directed, $w\not <
  w'$ must be forced by consecutive edges of the path that are
  oppositely-oriented and do not both lie in the same $2$-connected
  block. Therefore, $\Up$ and $\Down$ preserve $w\not<w'$. On the
  other hand, if there exists an $i_0$ such that
  $a_{i_0}\not<b_{i_0}$, then this (non)relation is preserved in one
  of $\Up$ or $\Down$. That extension preserves $w\not < w'$, since
  any directed path from $w$ to $w'$ would have to pass through the
  points $a_i$ and $b_i$, but there is no directed path between them
  in that extension. Therefore, we have shown $w\not<w'$ in at least of
  $\Up$ and $\Down$.
  \end{proof}

  As we combine the three preceding lemmas to prove our main theorem,
  we will reduce to a poset with an outerplanar cover graph. The
  following result guarantees that such posets have small dimension.

  \begin{theorem}[Felsner, Trotter, and Wiechert
    \cite{Fel-Tro-Wie-13-u}]
\label{thm:outerplanar}
If a poset\/ $P$ has an outerplanar cover graph, then\/ $\dim(P)\leq 4$.
\end{theorem}

We are finally ready to state the full version of our main theorem.

\begin{theorem}\label{T:full-main}
  Let\/ $P$ be a poset with cover graph\/ $G$. If\/ $G$ is\/
  $\mathcal{F}$-minor free, then\/ $\dim(P)\leq 17$.
\end{theorem}

\begin{proof}
  Begin by fixing a canonical embedding of $G$ in the plane and, as in
  Lemma~\ref{lem:subdivision}, let $Z$ be the collection of points
  that are not on the unbounded face of $G$ and are neither minimal
  nor maximal in $P$. By Lemma~\ref{lem:subdivision}, we know that
  $\dim(P)\leq 2\dim(P-Z)+1$. We now claim that $\dim(P-Z)\leq 8$,
  which will prove the theorem.  

  Applying Lemmas~\ref{lem:up-down} and \ref{lem:p-is-up-down} to
  $P-Z$, we find that $P-Z$ has two extensions $\Up$ and $\Down$ for
  which $P-Z=\Up\cap \Down$. Furthermore, since $P-Z$ does not contain
  any unidirected ears, the process of constructing $\Up$ and $\Down$
  cannot introduce unidirected ears, and the comparabilities added to
  form $\Up$ and $\Down$ turn beak peaks into vertices of degree $1$
  in the cover graphs, we have that $\Up$ and $\Down$ have outerplanar
  cover graphs. Therefore, by Theorem~\ref{thm:outerplanar}, there are
  realizers $\mathcal{R}_\Up$ and $\mathcal{R}_\Down$ of $\Up$ and
  $\Down$, respectively, with
  $|\mathcal{R}_\Up|,|\mathcal{R}_\Down|\leq 4$. Since
  $P-Z=\Up\cap\Down$, we know that $\mathcal{R}_\Up
  \cup\mathcal{R}_\Down$ is a realizer of $P-Z$. Therefore, $\dim(P-Z)\leq
  8$ and $\dim(P)\leq 17$.
\end{proof}

To obtain Theorem~\ref{T:weak-main}, we now note that if $P$ is a
poset with cover graph $G$ of pathwidth at most $2$, then $G$ is
$\mathcal{F}$-minor free, so Theorem~\ref{T:full-main} implies
$\dim(P)\leq 17$. It is natural to wonder whether the bound of
Theorem~\ref{T:full-main} is best possible. We have no reason to
believe the result is optimal and suspect it may be possible to
reduce the bound to $4$ with more work. That would be best possible,
as Felsner, Trotter, and Wiechert give a $4$-dimensional poset having
cover graph with pathwidth $2$ in \cite{Fel-Tro-Wie-13-u}.

We also note that Trotter \cite{Tro-per-13} has subsequently made an
observation regarding the relationship between dimension and the block
structure of the cover graph, making it possible to drop $T_3$, $T_4$,
and $T_5$ from the list of forbidden minors. However, that approach
leads to a weaker bound on the dimension than the one we
offer here.

\section{Standard Examples and Treewidth}

A second question posed in \cite{Jor-Mic-Mil-Tro-Wal-Wan-13-u} remains open.

\begin{question}\label{ques:tw2}
  Is there a constant $d$ such that if $P$ is a poset with cover
  graph $G$ and $\tw(G)\leq 2$, then $\dim(P)\leq d$?
\end{question}

The following theorem provides some weak evidence for an
affirmative answer to this question, since the theorem implies
that if the answer to Question~\ref{ques:tw2} is ``no'', a
counterexample cannot be constructed using large standard examples.

\begin{theorem}\label{T:S5}
  If\/ $P$ is a poset that contains the standard example\/ $S_5$ as a
  subposet, then the cover graph of\/ $P$ has treewidth at least 3.
\end{theorem}

\begin{proof}
  Since $\tw(K_4) =3$, it will suffice to show that the cover graph of
  $P$ has a $K_4$-minor. (In fact, more is true, in that $K_4$ is the
  only forbidden minor required to characterize graphs of treewidth
  $2$.) Since the notions of containing a $K_4$-minor and containing
  $K_4$ as a topological minor are equivalent, we use an approach that
  blends both techinques by seeking branch sets of a $K_4$ minor and
  joining them by internally disjoint paths. To aid in exposition, we
  will not fully specify the branch sets. Instead, we will refer to
  vertices or sets of vertices as being \emph{corners} of the $K_4$
  minor if they lie in distinct branch sets. We denote
  a path between any two comparable elements $x$ and $y$ such that the
  path represents a maximal chain between $x$ and $y$ in $P$ by $P(x,y)$. 

  Let $\set{a_1, \ldots, a_5}$ and $\set{b_1, \ldots, b_5}$ be
  elements of the subposet of $P$ isomorphic to $S_5$ with the
  standard ordering, that is, $a_i < b_j$ if and only if $i \neq j$.
  We first restrict our attention to the copy of $S_3$ determined by
  $\set{a_1,a_2,a_3,b_1,b_2,b_3}$.  In this context, fix $c_i$ as one
  of the maximal elements in $U[a_i] \cap D[b_{i+1}] \cap D[b_{i+2}]$
  where the subscripts are interpreted cyclically among $\set{1,2,3}$.
  Notice that $\set{c_1,c_2,c_3}$ is an antichain
  in $P$ since $a_i$ is incomparable to $b_i$ for all $i$. In a
  similar manner, fix $d_i$ as a minimal element in $U[c_{i+1}] \cap
  U[c_{i+2}] \cap D[b_i]$.
  Thus the poset $P$
  contains four (not necessarily disjoint) antichains
  $\set{a_1,a_2,a_3}, \set{b_1,b_2,b_3}, \set{c_1,c_2,c_3},$ and
  $\set{d_1,d_2,d_3}$ together with paths
  $P(a_i,c_i)$ and $P(d_i,b_i)$ for $i \in \set{1,2,3}$ and paths
  $P(c_i,d_j)$ for $i,j \in \set{1,2,3}$ with $i\neq j$.  
  See Figure
  \ref{F:standard3}. It is a straightforward, but tedious argument, to
  verify that these paths are all internally disjoint.  We call the subposet on these elements $S$.

  \begin{figure}[h]
    \centering
    \begin{overpic}[scale=0.6]{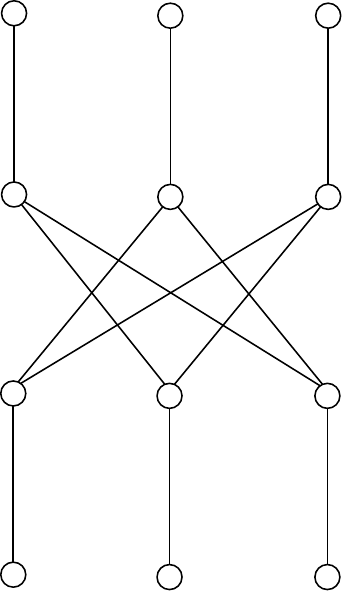}
      \put(-1,-7){$a_1$}
      \put(27,-7){$a_2$}
      \put(51,-7){$a_3$}
      \put(-1,103){$b_1$}
      \put(27,103){$b_2$}
      \put(51,103){$b_3$}
      \put(-7,23){$c_1$}
      \put(19,23){$c_2$}
      \put(57,23){$c_3$}
      \put(-7,71){$d_1$}
      \put(19,71){$d_2$}
      \put(57,71){$d_3$}
    \end{overpic}
    \caption{The subposet $S$ with vertices internal to chains/paths not shown.}
    \label{F:standard3}
  \end{figure}
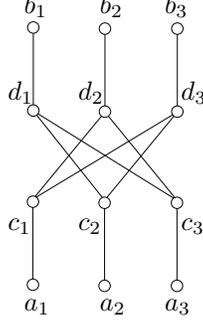

  After noting that $P(c_i,d_j)$ is internally disjoint from $P(c_{i'},d_{j'})$
  when $(i,j)\neq (i',j')$, it is easy to see that
  \[P(c_1,d_2),P(d_2,c_3),P(c_3,d_1),P(d_1,c_2),P(c_2,d_3),P(d_3,c_1)\]
  is a cycle in the cover graph of $P$. We denote this cycle by $C$. Thus, if any element $x$ of
  the poset is connected to this cycle by three paths intersecting
  only at $x$,
  then the cover graph contains a $K_4$-minor, as desired. Noting that
  $a_4 < b_1,b_2,b_3$ we now consider the relationship between $a_4$
  and $S$. Suppose first that $a_4$ is not less than any element of
  $\set{c_1,c_2,c_3}$. By our definitions, every element of $C -
  \set{c_1,c_2,c_3}$ is less than precisely one element of
  $\set{b_1,b_2,b_3}$. Hence, there exist three paths $P_1,P_2,P_3$ in
  the cover graph from $a_4$ to $C$. (Note that these paths may use
  the paths $P(b_i,d_i)$ if $a_4$ is not less than some of the $d_i$.)
  Each $P_i$ enters $C$ at a distinct point, creating a
  $K_4$-minor.

  Therefore, we may assume that $a_4$ is less than one element of
  $\set{c_1,c_2,c_3}$, say $c_1.$ By a similar argument, we may assume
  $b_4$ is greater than an element of
  $\set{d_1,d_2,d_3}$. Furthermore, since $b_4$ is incomparable to
  $a_4$ while $d_2$ and $d_3$ are comparable to $c_1$, our assumption
  that $a_4 < c_1$ forces $d_1$ to be the element of $\{d_1,d_2,d_3\}$
  that is less than $b_4$. Note that the incomparability between $a_4$
  and $b_4$ implies that $a_4$ is incomparable to $c_2$ and $c_3$ and
  $b_4$ is incomparable to $d_2$ and $d_3$. Additionally, there is a
  vertex $\beta_4$ on $P(d_1,b_1)$ such that $\beta_4 < b_4$ and a
  vertex $\alpha_4'$ on $P(a_1,c_1)$ such that $a_4 <
  \alpha_4'$.
  Since $a_4 < b_1$ and $a_4$ is incomparable to $b_4$, there is some
  element $\alpha_4$ on $P(d_1,b_1)$ with $\alpha_4>\beta_4$ and
  $a_4 < \alpha_4$. Similarly, there is an element $\beta_4'$ on
  $P(a_1,c_1)$ with $\beta_4'<\alpha_4'$ and $\beta_4' < b_4$. See
  Figure \ref{F:standard_wing} for an illustration of the relationship
  between these points. In a similar manner, we can find a
  $j \in \set{1,2,3}$ and elements $\beta_5$ on $P(d_j,b_j)$ and
  $\beta_5'$ on $P(a_j,c_j)$ such that $\beta_5,\beta_5' < b_5$. There
  are also elements $\alpha_5, \alpha_5' > a_5$ such that
  $\alpha_5 > \beta_5$ on $P(d_j,b_j)$ and $\alpha_5' > \beta_5'$ on
  $P(a_j,c_j)$. If there are multiple choices for $\beta_i$,
  $\beta'_i$, $\alpha_i$, and $\alpha'_i$ that satisfy all these
  requirements, we choose $\beta_i$ and $\beta'_i$ to be maximal and
  $\alpha_i$ and $\alpha'_i$ to be minimal among the possible choices.
  By our definitions of the $c_i$ and $d_j$, it is straightforward,
  but tedious, to verify that $P(a_4,\alpha_4), P(a_4,\alpha_4'),$ and
  $P(\beta_4,b_4)$ are internally disjoint from $S$.  Further,
  $P(\beta_4',b_4)$ is internally disjoint from $S$ except for
  possibly $P(c_2,d_3)$ and $P(c_3,d_2)$.

Suppose then that $P(\beta_4',b_4)$ intersects both $P(c_2,d_3)$ and
$P(c_3,d_2)$. Let $K_4-e$
  denote the graph that results from deleting any edge from $K_4$.  It
  is easy to see that there is a $(K_4-e)$-minor with corners $c_1$,
  $d_1$, and the two intersection points of the path $P(\beta'_4,b_4)$
  with $P(c_2,d_3)$ and $P(c_3,d_2)$.  (Note that this minor can be
  formed using only $C$ and the part of $P(\beta'_4,b_4)$ between
  $P(c_2,d_3)$ and $P(c_3,d_2)$.) The missing connection to
  complete the $K_4$-minor is the edge between $c_1$ and $d_1$.
  However, as $P(d_1,\alpha_4)P(\alpha_4,a_4)P(a_4,\alpha_4')P(\alpha'_4,c_1)$ is disjoint from the
  cycle $C$ and $P(\beta_4',b_4)$, this completes the $K_4$-minor.
  Thus we may assume that $P(\beta_4',b_4)$ intersects only one of
  $P(c_2,d_3)$ and $P(c_3,d_2)$.  Without loss of generality, suppose
  the intersected path is
  $P(c_2,d_3)$ and let $z$ be the maximal point of intersection.  We note now
  that there is a cycle formed by 
\[
P(z,b_4)P(b_4,\beta_4)P(\beta_4,\alpha_4)P(\alpha_4,a_4)P(a_4,\alpha_4')P(\alpha_4',c_1)P(c_1,d_3)P(d_3,z).\]
Furthermore, the point $d_1$ has three distinct paths to this cycle,
forming a $K_4$ minor.  Thus the paths $P(a_4,\alpha_4)$,
$P(a_4,\alpha_4')$, $P(b_4,\beta_4)$, and $P(b_4,\beta_4')$ are all
internally disjoint from $S$ as shown in Figure \ref{F:standard_wing}.

\begin{figure}[h]
    \centering\vspace{7pt}
    \begin{overpic}[height=102pt]{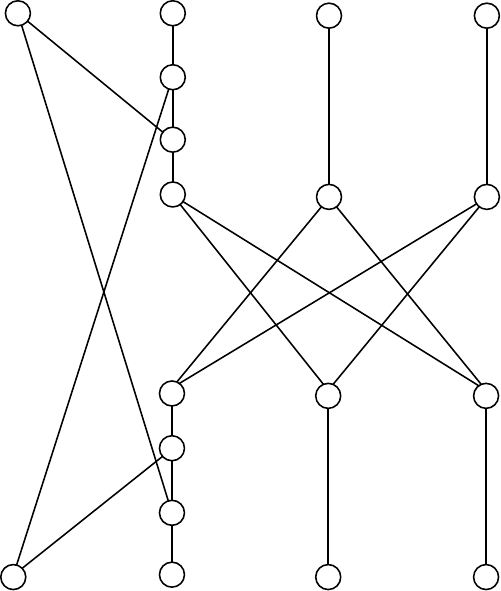}
      \put(26,-7){$a_1$}
      \put(54,-7){$a_2$}
      \put(78,-7){$a_3$}
      \put(-1,-7){$a_4$}
      \put(26,103){$b_1$}
      \put(54,103){$b_2$}
      \put(78,103){$b_3$}
      \put(-1,103){$b_4$}
      \put(33,31){$c_1$}
      \put(59,31){$c_2$}
      \put(87,31){$c_3$}
      \put(33,66){$d_1$}
      \put(59,66){$d_2$}
      \put(87,66){$d_3$}
      \put(33,75){$\beta_4$}
      \put(33,85){$\alpha_4$}
      \put(33,10){$\beta_4'$}
      \put(33,21){$\alpha_4'$}      
    \end{overpic}
    \caption{Expanding $S$ by adding $a_4,b_4,\alpha_4,\beta_4,\alpha_4',\beta_4'$.}
    \label{F:standard_wing}
  \end{figure}
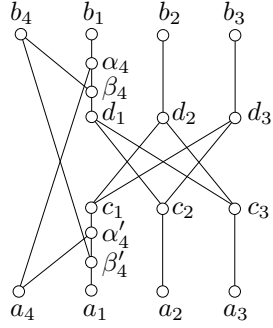

  We consider the cases where $j \neq 1$ and $j= 1$
  separately. (Recall that $j$ is the index such that $\beta_5\in
  P(d_j,b_j)$.) For
  the former, suppose without loss of generality that $j = 3$, as
  depicted in Figure~\ref{fig:double-wing}.
  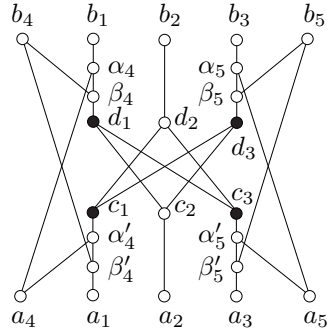
\begin{figure}
    \centering\vspace{7pt}
    \begin{overpic}[height=102pt]{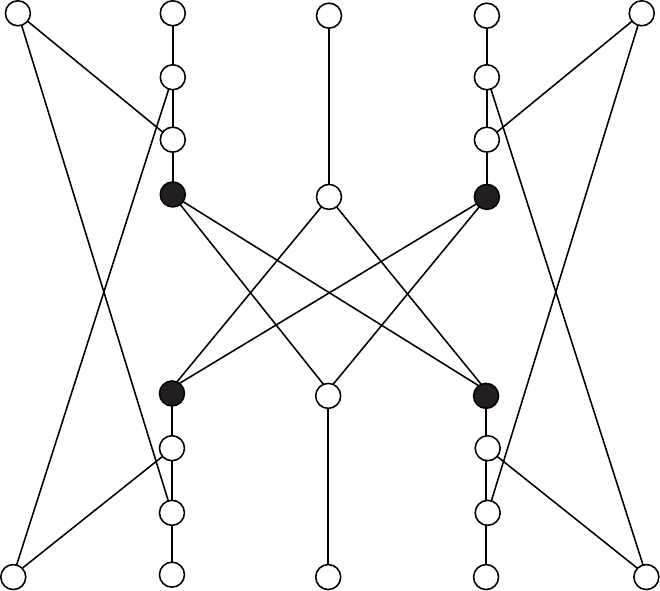}
      \put(24,-7){$a_1$}
      \put(48,-7){$a_2$}
      \put(71,-7){$a_3$}
      \put(-1,-7){$a_4$}
      \put(96,-7){$a_5$}
      \put(24,93){$b_1$}
      \put(48,93){$b_2$}
      \put(71,93){$b_3$}
      \put(-1,93){$b_4$}
      \put(96,93){$b_5$}
      \put(31,30){$c_1$}
      \put(53,30){$c_2$}
      \put(72,34){$c_3$}
      \put(31,59){$d_1$}
      \put(53,59){$d_2$}
      \put(72,49){$d_3$}
      \put(31,67){$\beta_4$}
      \put(31,76){$\alpha_4$}
      \put(31,9){$\beta_4'$}
      \put(31,20){$\alpha_4'$}      
      \put(61,67){$\beta_5$}
      \put(61,76){$\alpha_5$}
      \put(61,9){$\beta_5'$}
      \put(61,20){$\alpha_5'$}      
    \end{overpic}
    \caption{The case where $a_5$ and $b_5$ attach to different paths
      than $a_4$ and $b_4$.}
    \label{fig:double-wing}
  \end{figure}
  In this case, if the following six paths are internally disjoint,
  they form a $K_4$-minor with corners $c_1$, $d_1$, $c_3$, and $d_3$:
  \begin{itemize}
  \item $P(d_1,c_2)P(c_2,d_3)$,
  \item $P(d_3,\alpha_5)P(\alpha_5,a_5)P(a_5,\alpha_5')P(\alpha_5',c_3)$,
  \item $P(c_3,d_2)P(d_2,c_1)$,
  \item $P(c_1,\alpha_4')P(\alpha_4',a_4)P(a_4,\alpha_4)P(\alpha_4,d_1)$,
  \item $P(d_1,c_3)$, and
  \item $P(d_3,c_1)$.
  \end{itemize}
  The internal disjointness of each pair of the paths above is clear
  with the possible exception of the second path and the fourth
  path. However, if these paths fail to be disjoint, their
  intersection point has $3$ paths to distinct vertices of $C$,
  creating a $K_4$-minor.

  The most delicate part of our argument remains in the case where
  $j=1$. We consider now the paths that enter
  $P(d_1,b_1)$. Specifically, we examine the relationships between
  $P(a_4,\alpha_4), P(a_5,\alpha_5), P(b_4,\beta_4),$ and
  $P(b_5,\beta_5)$. The paths entering $P(a_1,c_1)$ featuring the
  $\alpha_i'$ and $\beta_i'$ will interact identically by duality.  It
  is clear that $P(a_4,\alpha_4)$ and $P(b_4,\beta_4)$ do not
  intersect, as otherwise $a_4 < b_4$. (A similar argument applies to
  $P(a_5,\alpha_5)$ and $P(b_5,\beta_5)$.) Suppose then that
  $P(a_4,\alpha_4)$ and $P(b_5,\beta_5)$ intersect at some point $x$,
  while $P(a_5,\alpha_5)$ and $P(b_4,\beta_4)$ do not
  intersect. Furthermore, if the paths $P(a_4,\alpha_4)$ and
  $P(b_5,\beta_5)$ intersect more than once, we will assume that $x$
  is the minimal such intersection (in terms of the poset).

  Now consider rerouting the path $P(d_1,b_1)$ through $x$. The
  new path will be the concatenation of $P(d_1,\beta_5)$,
  $P(\beta_5,x)$, $P(x,\alpha_4)$, and $P(\alpha_4,b_1)$. We then
  choose the new vertices $\hat\alpha_4$, $\hat\alpha_5$,
  $\hat\beta_4$, $\hat\beta_5$ appropriately, recalling that they are
  chosen to be maximal or minimal amongst possible options.  The new
  paths $P(a_4,\hat{\alpha}_4)$ and $P(b_5,\hat{\beta}_5)$ are
  internally disjoint by construction.  Suppose now that $\alpha_5$ is
  a element of the path $P(\beta_5,\alpha_4)$ and consider the cycle
  formed by $P(\beta_5,x)$, $P(x,\alpha_4)$, and
  $P(\alpha_5,\beta_5)$.  (See Figure~\ref{F:reroute-x}.)
  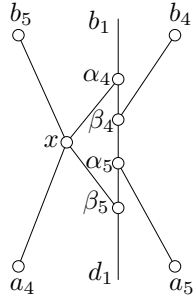
\begin{figure}[h]
    \centering
    \begin{overpic}[scale=0.6]{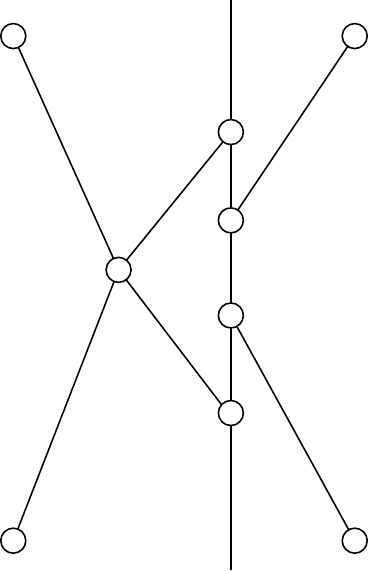}
      \put(-1,99){$b_5$}
      \put(60,99){$b_4$}
      \put(29,96){$b_1$}
      \put(-1,-4){$a_4$}
      \put(60,-4){$a_5$}
      \put(29,0){$d_1$}
      \put(12,51){$x$}
      \put(27,76){$\alpha_4$}
      \put(29,58){$\beta_4$}
      \put(28,43){$\alpha_5$}
      \put(27,27){$\beta_5$}
    \end{overpic}
    \caption{Rerouting $P(d_1,b_1)$ via $x$.}
    \label{F:reroute-x}
  \end{figure}
  Observe that there are three disjoint paths---namely, $P(x,a_4),
  P(\beta_5,d_1)$, and $P(\alpha_5,a_5)$---emanating from the
  cycle. Since $a_4,a_5,$ and $d_1$ all connect to the path
  $P(a_1,c_1)$, these three vertices are all in the same connected
  component after deleting the cycle. Therefore, we have found a
  $K_4$-minor.   In a similar manner, we may assume that $\beta_4$ is
  not on the path $P(\beta_5,\alpha_4)$.  Thus we have that
  $\hat{\alpha}_5 = \alpha_5$ and $\hat{\beta}_4 = \beta_4$,  and
  furthermore by our assumptions, the paths
  $P(a_5,\hat{\alpha}_5)$ and $P(\hat{\beta}_4,b_4)$ do not intersect.

  Now consider the case where, in addition, $P(a_5,\alpha_5)$ and
  $P(b_4,\beta_4)$ intersect at some point $y$, again choosing $y$ as
  the minimal intersection point. Since $\beta_5 < x < \alpha_4$,
  $\beta_4 < y < \alpha_5$, $\beta_4 < \alpha_4$, and $\beta_5 <
  \alpha_5$, we have that $\set{\beta_4,\beta_5} <
  \set{\alpha_4,\alpha_5}$. Since $a_i$ is incomparable to $b_i$ in
  the poset, we must have that $x$ and $y$ are incomparable as
  well. This implies that any intersection between $P(x,\beta_5)$ and
  $P(y,\beta_4)$ occurs at a point less than both $x$ and $y$ on these
  paths. Similarly, any intersection between $P(x,\alpha_4)$ and
  $P(y,\alpha_5)$ must be greater than both $x$ and $y$.  It
  is then easy to see that there is a $(K_4 - e)$-minor with corners
  $x,y,\set{\beta_4,\beta_5},$ and $\set{\alpha_4,\alpha_5}$, possibly
  adding intersection points between $P(x,\beta_5)$ and $P(y,\beta_4)$
  to $\set{\beta_4,\beta_5}$ and intersection points between
  $P(x,\alpha_4)$ and $P(y,\alpha_5)$ to
  $\set{\alpha_4,\alpha_5}$. The missing connection to complete a
  $K_4$-minor is between $x$ and $y$. However, notice that $x$ and $y$
  are connected by a path through $a_4$, $a_5$, and $P(a_1,c_1)$,
  giving the needed path to complete the minor.

  We are now able to make a fairly strong assumption about the
  pairwise intersections of
  $P(a_4,\alpha_4),P(a_5,\alpha_5),P(b_4,\beta_4),$ and
  $P(b_5,\beta_5)$. Of the six possible crossings, the only two that
  can occur are $P(a_4,\alpha_4)$ with $P(a_5,\alpha_5)$ and
  $P(b_4,\beta_4)$ with $P(b_5,\beta_5)$. Furthermore, these
  intersections imply that $\alpha_4 = \alpha_5$ or $\beta_4 =
  \beta_5$, respectively, by the maximality of the $\beta_i$ and
  minimality of the $\alpha_j$.

  Having established these intersection limitations (and the
  corresponding ones for the $\alpha_i'$ and $\beta_i'$), we consider
  the graph formed by contracting each of $P(a_i,\alpha_i)$,
  $P(a_i,\alpha_i')$, $P(b_i,\beta_i)$, and $P(b_i,\beta_i')$ for $i =
  4,5$ to a single edge. In fact, we go further and contract
  (arbitrarily) all the edges we can while ensuring that the
  $\alpha_i$, $\beta_i$, $\alpha_i'$, $\beta_i'$, $a_i$, and $b_i$ are
  not identified for $i=4,5$. Since it is possible for some of these
  vertices to have been equal at the outset, we are then left with a
  graph with at most $12$ vertices. (We refer to the vertices as
  having labels to allow that, for example, $\alpha_4$ and $\alpha_5$
  may refer to the same vertex.)  The resulting graph is built up from
  a path in which the vertices with labels $V =
  \set{\alpha_4,\alpha_5,\beta_4,\beta_5}$ appear consecutively, as do
  the vertices with labels $V' =
  \set{\alpha_4',\alpha_5',\beta_4',\beta_5'}$. In addition to this
  primary path, the graph resulting from the contraction contains a
  collection of $4$ paths of length $2$ (via the $a_i$ and $b_i$)
  connecting vertices in $V$ and $V'$.

  We now show that in all but one case (to be described later), this
  graph has a $K_4$-minor. Let $M'$ be the labels of the maximum
  vertices (with respect to the poset) of those with labels in
  $V'$ and similarly define $M$ as the labels of the minimal vertices of
  those with labels in $V$.
  \begin{figure}[h]
    \centering
    \begin{overpic}[scale=0.6]{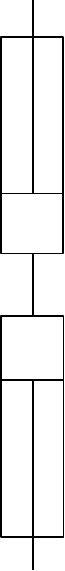}
      \put(11,20){$V'$}
      \put(11,73){$V$}
      \put(0,36){\scriptsize$M'$}
      \put(1,58){\scriptsize$M$}
    \end{overpic}
    \caption{Relation of $V$, $M$, $V'$, and $M'$ in the graph after contractions.}
    \label{F:M-block}
  \end{figure}
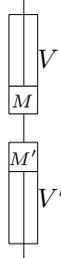
Since $M$
  and $M'$ each correspond to a single vertex, they cannot contain
  two labels with the same subscript and must respect the ordering on
  elements with the same subscript. Thus, $M \in \set{\set{\beta_4},\set{\beta_5},
    \set{\beta_4,\beta_5}}$ and $M' \in \set{\set{\alpha_4'},
    \set{\alpha_5'}, \set{\alpha_4',\alpha_5'}}$. We now construct the $K_4$-minor
  using $M$, $V -  M$, $M'$, and $V' -  M'$ as the
  corners. Figure~\ref{F:M-block} makes clear that (with appropriate
  contractions), $V -  M,M,M',V' -  M'$ is a path of
  length $3$. 

  To construct the $K_4$-minor, it suffices to show that there are
  connections between (1) $M$ and $V' - M'$, (2) $M'$ and $V - M$, and
  (3) $V - M$ and $V' - M'$. Since $\alpha_i$ and $\alpha'_i$ are
  connected via $a_i$ for $i=4,5$, and $\beta_i$ and $\beta'_i$ are
  connected via $b_i$ for $i=4,5$, the first two cases are immediately
  resolved because of the possible contents of $M$ and $M'$.
  Furthermore, for the third pair, the only situation where we do not
  immediately see a connection between $V - M$ and $V' - M'$ is when
  $M = \set{\beta_4,\beta_5}$ and $M' = \set{\alpha_4',\alpha_5'}$.

  In this case, the poset must contain the paths depicted in
  Figure~\ref{F:final-case}. However, in this case there is a path
  between $a_4$ and $b_5$ since $a_4<b_5$ in the poset. This path is
  not depicted in Figure~\ref{F:final-case}. If this path is disjoint
  from $C$ it is straightforward to verify that there is a $K_4$-minor
  with corners  $a_4$, $b_5$, $\alpha_4'$, and $\beta_5$.  Otherwise,
  as $a_i$ is incomparable to $b_i$ for $i = 4,5$, the path from $a_4$
  to $b_5$ can only intersect $C$ in the paths $P(c_2,d_3)$ or
  $P(c_3,d_2)$.  Without loss of generality, suppose the path from
  $a_4$ to $b_5$ intersects $C$ at $P(c_2,d_3)$.  Then there is a
  $K_4$-minor with corners $a_4$, $b_5$, $\beta_5$, and $P(c_2,d_3)$.  
  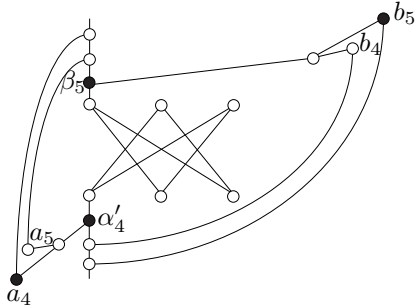
\begin{figure}[h]
    \centering
    \begin{overpic}[scale=0.6]{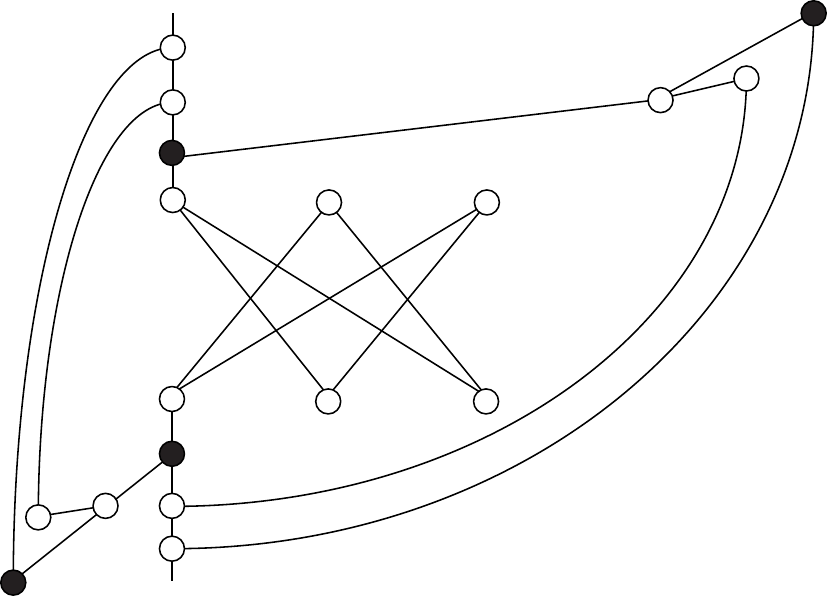}
      \put(-1,-4){$a_4$}
      \put(5,12){$a_5$}
      \put(101,70){$b_5$}
      \put(92,62){$b_4$}
      \put(23,16){$\alpha_4'$}
      \put(13,52){$\beta_5$}
    \end{overpic}
    \caption{There must also be a path from $b_5$ to $a_4$.}
    \label{F:final-case}
  \end{figure}
\end{proof}

We conclude this section by observing that for any $x\in S_5$, there
is a poset containing $S_5 -  x$ and having a cover graph
of treewidth 2. We show an example in Figure \ref{F:optimal}, with the
poset on the left and a redrawing of the cover graph on the
right. Notice that $S_5 - x$ is the subposet
formed by the elements other than $u$ and $v$. (Since the graph is clearly $K_4$-minor-free, it has treewidth
at most $2$.) This implies that Theorem~\ref{T:S5} is best
possible.
\begin{figure}[h]
  \centering
  \begin{overpic}[scale=0.6]{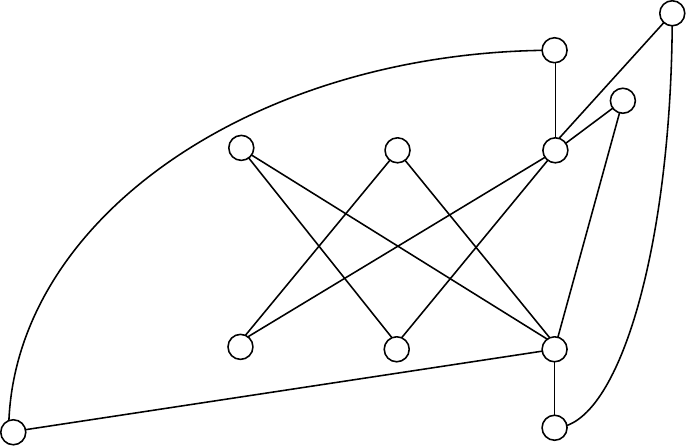}\put(83,13){$u$}\put(74,42){$v$}\end{overpic}\hspace{0.5in}\begin{overpic}[scale=0.6]{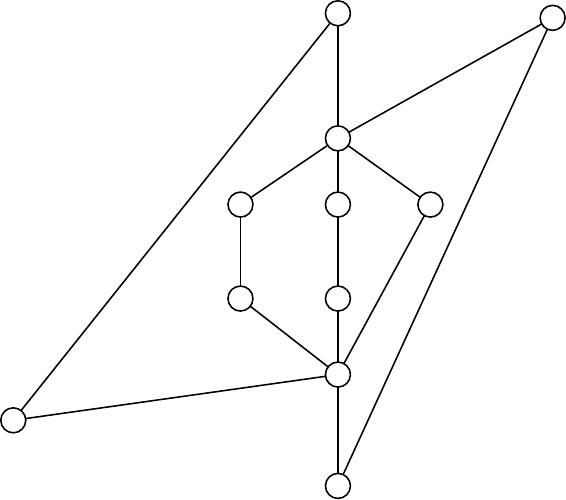}\put(52,12){$u$}\put(52,63){$v$}\end{overpic}
  \caption{A poset containing $S_5 - x$ (left) with cover
    graph $G$ (right) redrawn to help show $G$ does not contain a
    $K_4$-minor, and therefore $\tw(G)=2$.}
  \label{F:optimal}
\end{figure}

\section*{Update on Question~\ref{ques:tw2}}
While this paper was under review, Joret, Micek, Trotter, Wang, and
Wiechert announced that they have resolved Question~\ref{ques:tw2} in
the affirmative~\cite{joret:tw-2} with a bound on the dimension of 1276.

\section*{Acknowledgments}

We are grateful to the three anonymous referees for their careful
reading of the paper; their feedback significantly improved the final
version. We also thank one referee for bringing Spinrad's subdivision
paper \cite{spinrad:subdiv} to our attention.

\bibliographystyle{acm}
\bibliography{main}

\end{document}